\documentclass[a4paper, 11pt, leqno]{article}

\usepackage{amssymb}
\usepackage{amsmath}
\usepackage{amsthm}
\usepackage{graphicx, color, xcolor}
\usepackage{helvet}

\usepackage[normalem]{ulem}

\usepackage[utf8]{inputenc}
\input xypic.sty

\definecolor{sepia}{rgb}{0.4,0,0.1}

\def\parcial#1#2{\frac{\partial #1}{\partial#2}}

\def\nn{\nonumber}

\def\exp{\operatorname{exp}}

\def\flecha{\longrightarrow}

\def \oN{{\overline \nabla}}

\def \wN{{\widetilde \nabla}}

\def \hN{{\widehat\nabla}}

\def\oH{{\overline H}}

\def\gR{\gamma_R}
\def\g{\gamma}
\def\FR{F_R}

\def\ot{{\overline t}}

\def\ema{M^{n-q}_\lambda}

\def\emaz'd{M^{n-q-d}_{\lambda \frak z'}}

\def\emafz'd{M^{n-q-d}_{\lambda f(z')}}

\def\ema0{M^{n-1}_{\lambda 0}}

\def\ce{\mathbb C}

\def\<{\langle}
\def\>{\rangle}
\def\({\left(}
\def\){\right)}

\def\re{\mathbb R}
\def\${\sqrt\lambda}
\def\ce{\mathbb C}

\def\R{\mathbb R}
\def\C{\mathbb C}

 at 20truept
 at 24truept
 at 14truept
 at 16truept
\newtheorem{teor}{Theorem}

\newtheorem{prop}[teor]{Proposition}
\newtheorem{lema}[teor]{Lemma}

\newtheorem{nota}[teor]{Remark}
\newtheorem{coro}[teor]{Corollary}

\newtheorem{theoremL}{Theorem}

\numberwithin{teor}{section}

\begin{document}

\title{Evolution by mean curvature flow \\ of Lagrangian spherical surfaces in complex Euclidean plane}

\author{Ildefonso Castro$^*$, Ana M. Lerma\thanks{Research partially supported by  a MEC-FEDER  grant MTM2014-52368-P.} \ and Vicente Miquel\thanks{Research partially supported by  the
MEC-FEDER  Project MTM2016-77093-P and the Generalitat Valenciana Project  PROMETEOII/2014/064.}}

\date{}
\maketitle{}

\begin{abstract} We describe the evolution under the mean curvature flow of embedded Lagrangian spherical surfaces in the complex Euclidean plane $\C^2$.
In particular,  we answer the Question 4.7 addressed in \cite{Ne10b} by A.\ Neves about finding out a condition on a starting Lagrangian torus in $\C^2$ such that the corresponding mean curvature flow becomes extinct at finite time and converges after rescaling to the Clifford torus. 
{\it 2010 Mathematics Subject Classification:\,} Primary 53C44, 53C40; Secondary 53D12.

\end{abstract}

\section{Introduction}

Let $F_0 : M^n \rightarrow \re^m$ be an immersion of a
compact manifold of dimension $n\geq 2$ into Euclidean space. The
mean curvature flow with initial condition $F_0$ is a smooth
family of immersions $F:M\times [0,T) \rightarrow \re^m$ satisfying
\begin{equation}\label{MCF}
\frac{\partial}{\partial t} F(p,t) = H(p,t), \ p\in M, \, t\geq 0;
\quad F(\cdot,0)=F_0, \tag{MCF}
\end{equation}
where $H(p,t)$ is the mean curvature vector of the submanifold
$M_t=F(M,t)$ at $p$. It is well-known that (\ref{MCF}) is a
quasilinear parabolic system that is invariant under
reparametrizations of $M$ and isometries of the ambient space and
short-time existence and uniqueness is guaranteed, being
$T<\infty$ the maximal time of existence.

The first classical works in this topic studied the evolution of
hypersurfaces by their mean curvature. We emphasize Huisken's
paper \cite{Hu84} on the flow of convex surfaces into spheres,
proving that if the initial hypersurface is uniformly convex, then
the mean curvature flow converges to a round point in finite time.
That is, the shape of $M_t$ approaches the shape of a sphere very
rapidly and no singularities will occur before the hypersurfaces
$M_t$ shrink down to a single point after a finite time. Recently,
mean curvature flow of higher codimension submanifolds has also
received interest by many authors who have paid attention mainly
to graphical submanifolds and symplectic or Lagrangian
submanifolds. We recall that Huisken's monotonicity formula
\cite{Hu90}, relating the formation of singularities to
self-shrinking solutions of the mean curvature flow, also applies
in any codimension. Concretely, the so-called Type I singularities
forming in Euclidean space look like self-similar contracting
solutions after an appropriate rescaling procedure. According to
\cite{Sm11}, this type of singularities usually occur when there
exists some kind of pinching of the second fundamental form.
Andrews and Baker \cite{AB10} proved a convergence theorem for the
mean curvature flow of closed submanifolds satisfying suitable
pinching condition and showed that such submanifolds contract to
round points.

In this paper we are interested in the class of Lagrangian immersions in complex Euclidean space
$\ce^n\!\equiv\!\re^{2n}$, which is a preserved class under the mean curvature flow. We notice that there do not exist
Lagrangian self-shrinking spheres (see \cite{CL14} or \cite{Sm11} and references therein) and, in addition, Smoczyk
showed that the class of smooth closed Lagrangian immersions in $\ce^n $ is not $\delta$-pinchable for any $\delta$
(see \cite[Section 4.1]{Sm11}). The authors do not know any available result regarding convergence of compact
Lagrangians in $\ce^n$. In fact, the following problem was posed by André Neves \cite[Question 7.4]{Ne10b} as a
Lagrangian analogue of Huisken's classical result \cite{Hu84} for the mean curvature flow of convex spheres:
\begin{quote}
{\it Find a condition on a Lagrangian torus in $\ce^2$, which
implies that the Lagrangian mean curvature flow $(M_t)_{0<t<T}$
will become extinct at time $T$ and, after rescale, $M_t$
converges to the Clifford torus.}
\end{quote}

Our contribution to this problem is the following main result.

\begin{theoremL}\label{TheoremA}
Let $M_0$ be an embedded Lagrangian compact surface of $\mathbb C^2$ which is contained in some hypersphere $\mathbb S^3(R_0)$ of radius $R_0>0$. Then the mean curvature flow \eqref{MCF} with initial condition $M_0$ has a unique solution defined on a maximal interval $[0,T)$, $T\leq R_0^2/4$. In addition:
\begin{itemize}
\item[(a)] If $M_0$ divides $\mathbb S^3(R_0)$ in two connected components of equal volume, then $T=R_0^2/4$, the limit $M_T$ of the evolving surfaces $M_t$ when $t\to T$ is a point and, after rescaling the flow by multiplication by $1/\sqrt{R_0^2-4 t}$, the limit is a Clifford torus in $\mathbb S^3 \!:= \mathbb S^3(1)\subset \ce ^2$.
\item[(b)] If $M_0$ divides $\mathbb S^3(R_0)$ in two connected components of different volumes (being $2 \pi A_0 R_0^3$ the lowest volume), then $T=A_0 R_0^2 /2\pi$, the limit $M_T$ of the evolving surfaces $M_t$ when $t\to T$ is a circle of radius $R_0 \sqrt{1-2 A_0/\pi}$ and, after rescaling the flow by multiplication by $\sqrt{A_0(R_0^2 - 4 t)/(A_0R_0^2-2\pi t)}$, the limit is a cylinder in $\re^3\subset \ce^2$. 
\end{itemize}
\end{theoremL}

In Proposition \ref{MHC} (see also \cite[Corollary 1]{We91}) we show that any compact Lagrangian surface of complex Euclidean plane contained in some hypersphere must be the preimage of a spherical closed curve by the corresponding Hopf fibration, providing in general an immersed torus that was called a Hopf torus by Pinkall in \cite{Pi85}.  As we shall see in the proof of Theorem \ref{TheoremA}, $A_0$ coincides with the area enclosed by the spherical curve $\pi (M_0/R_0) $, projection of $(1/R_0) M_0 \subset \mathbb S^3$ on $\mathbb S^2(1/2)$ by the Hopf fibration $\pi:\mathbb S^3\rightarrow \mathbb S^2(1/2) $. The isometry type of the torus $M_0$ depends not only on the length of the spherical curve $\pi (M_0/R_0) $ but also on the enclosed area $A_0$. It was proved in \cite{Pi85} that a Hopf torus $M_0$ is a critical point of the Willmore functional if and only if its corresponding spherical curve is an elastic curve.

Part (a) of Theorem \ref{TheoremA} is our answer to the Neves question quoted before for a Lagrangian embedded torus $M_0$.
We point out that the hypotheses on $M_0$ established in Theorem \ref{TheoremA} are preserved by the mean curvature flow (see Lemma \ref{dAt}).
Thinking of the shape of the closed spherical elastic curves, $M_0$ could be a Willmore torus. We remark that the first two
authors provided in \cite{CL14} four rigidity results for the Clifford torus in the class of compact self-shrinkers for the Lagrangian
mean curvature flow.

All the singularities appearing in Theorem \ref{TheoremA} are of Type I (see Remark \ref{notsingI}). \subsection{The ideas behind the main results} 
We now expose some ideas showing that the evolutions considered in Theorem \ref{TheoremA}  are natural in some geometric
sense since they (and some other studied in \cite{GSSmZ07}, \cite{Ne07}  and \cite{Ne10a}) can be regarded as evolutions
related with geometric flows of plane and spherical curves.

Let $\alpha_0:I_1\rightarrow \C^*$ be a regular plane curve and $\gamma_0:I_2 \rightarrow\mathbb S^2 (1/2) $
be a regular spherical curve in $\mathbb C^2$, where $I_1$ and $I_2$ are intervals in $\R$. Let
\[
F_0: I_1 \times I_2 \subseteq \R^2 \longrightarrow \C^2,  \qquad F_0(x,y) = \alpha_0 (x) \tilde{\gamma}_0 (y),
\]
with $\tilde\gamma_0:I_2 \rightarrow\mathbb S^3\subset \C^2 $ a horizontal lift of $\gamma_0$ via the Hopf fibration
$\pi: \mathbb S^3 \rightarrow \mathbb S^2(1/2)$. We denote by $\langle \cdot , \cdot \rangle$ and $J$ the Euclidean metric and the complex structure in $\ce ^2$ and consider simultaneously a one-parameter family of plane curves
\[
 \alpha=\alpha(x,t) \in \C^*, \, t\geq 0, \ \mathrm{with \  }
\alpha(x,0)=\alpha_0(x), \, x\in I_1,
\]
and a one-parameter family of spherical curves
\[
 \gamma=\gamma(y,t)\in \mathbb S^2(1/2), \, t\geq 0, \ \mathrm{with \ }
\gamma(y,0)=\gamma_0(y), \, y\in I_2,
\]
and define (see \cite{RU98}) the Lagrangians
\begin{equation}\label{Ft}
F=F(x,y,t)=\alpha(x,t)\tilde{\gamma}(y,t), \quad t\geq 0, \, (x,y)\in I_1\times I_2 \subseteq \R^2,
\end{equation}
where $\tilde{\gamma}=\tilde{\gamma}(y,t)\in\mathbb S^3\subset\mathbb{C}^2$ is a horizontal lift of $\gamma=\gamma(y,t)$ via the Hopf
fibration $\pi: \mathbb S^3 \rightarrow \mathbb S^2(1/2)$.

It is clear that $F(x,y,0)=F_0(x,y)$. Our goal is to analyse the possible evolutions of $\alpha$ and $\gamma$ in order
to $F$ be a solution of (\ref{MCF}).  Using \cite{RU98} and the Lagrangian character of each $F_t:=F(\cdot,\cdot,t)$,
$t\geq 0$, it is not difficult to get that $F$ is a solution of (\ref{MCF}) if and only if the following two equations
(corresponding to the normal directions $J(F_t)_x$ and $J(F_t)_y$) are satisfied:
\begin{equation}\label{Fsolx}
\left\langle \frac{\partial \alpha}{\partial t}, i\alpha_x \right\rangle   + \langle \alpha ,\alpha_x  \rangle
\left\langle \frac{\partial \tilde{\gamma}}{\partial t}, J \tilde \gamma \right\rangle = |\alpha_x|\kappa_\alpha +
\frac{\langle \alpha_x ,i\alpha \rangle}{|\alpha|^2}
\end{equation}
and
\begin{equation}\label{Fsoly}
 |\alpha|^2 \left\langle \frac{\partial \tilde{\gamma}}{\partial t}, J {\tilde \gamma_y}\right\rangle
 =|{\tilde\gamma_y}| \kappa_{\tilde\gamma}.
\end{equation}
In formulae \eqref{Fsolx} and \eqref{Fsoly} and in the rest of this section, subscript $x$ (resp.\ $y$) means derivative respect to $x$ (resp.\ $y$) and $\kappa $ will always denote curvature of the corresponding curve along the paper.  
Looking at (\ref{Fsoly}) we distinguish two complementary cases:

{\em Case (i):} there is no (normal) evolution for $\tilde{\gamma}=\tilde{\gamma}(y,t)$ (and hence for
$\gamma=\gamma(y,t)$) and so $\tilde \gamma$ (and $\gamma$) must be a static geodesic, say $$ \tilde{\gamma}(y,t)=(\cos y,
\sin y)\in\mathbb{C}^2, \, \forall t\geq 0 .$$ 

Then equation (\ref{Fsolx}) can be easily rewritten as
\[
\left( \frac{\partial \alpha}{\partial t} \right)^\bot =
\vec{\kappa}_\alpha - \frac{\alpha^\bot}{|\alpha|^2},
\]
where $\vec{\kappa}_\alpha$ is the curvature vector of $\alpha $
and $\alpha^\bot$ denotes the normal component of $\alpha$.
Putting this information in (\ref{Ft}) we arrive at the evolution
studied in \cite{GSSmZ07}, \cite{Ne07}  and \cite{Ne10a}.

{\em Case (ii):} necessarily $|\alpha|$ only depends on time
variable $t$, say $R(t):=|\alpha|$. This means that the evolution
of $\alpha = \alpha(x,t)$ consists of concentric circles centered
at the origin and, up to reparametrizations, it can be given by
$\alpha(x,t)=R(t)e^{ix}$. Now (\ref{Fsolx}) translates into a
simple o.d.e.\ for $R(t)$, concretely $-R \, \mathrm{d}R/\mathrm{d}t = 2$, whose
general solution is $R(t)=\sqrt{R(0)^2-4t}$. Putting this in
(\ref{Ft}), we get that in this case $F$ can be written as
\begin{equation}\label{FtNEW}
F(x,y,t)=\sqrt{R_0^2-4t}\, e^{ix}\, \tilde{\gamma}(y,t), \quad 0\leq t < \frac{R_0^2}{4},
\end{equation}
with $R_0=R(0)$ and where $\tilde{\gamma}(y,t)$ satisfy now the equation, coming from (\ref{Fsoly}), given by
\begin{equation}\label{flowLegendre}
 \left \langle \frac{\partial \tilde{\gamma}}{\partial t}, \frac{J {\tilde \gamma_y}}{|{\tilde \gamma_y}| }
 \right\rangle
 =\frac{\kappa_{\tilde \gamma}}{R_0^2-4t}.
\end{equation}
Using that the Hopf fibration $\pi$ is a Riemannian submersion, we rewrite (\ref{flowLegendre}) as
\begin{equation}\label{flowspherical}
\left\langle\frac{\partial \gamma}{\partial t} , \frac{\gamma \times  \gamma_y}{| \gamma_y|} \right\rangle
=\frac{2\kappa_{\gamma}}{R_0^2-4t},
\end{equation}
where $\times $ denotes the cross product in $\R^3$. 
We will check in Section \ref{TA} that (\ref{flowspherical}) is essentially the curve shortening flow in $\mathbb
S^2(1/2)$. The relation between this flow and the corresponding flow \eqref{FtNEW} of the initial Lagrangian surface will lead to different situations and their study in depth allows us to prove Theorem \ref{TheoremA}  in Section \ref{TA}.

\bigskip
\noindent {\bf Acknowledgments:}  The authors wish to thank A.~Neves, K.~Smoczyk and M.-T.~Wang for interesting conversations related to this paper. 

The authors would like to thank the referee for the very careful
review and for providing a number of valuable comments and suggestions.

\section{Preliminaries}

\subsection{The geometry of Hopf tori in the Hopf fibration}\label{hopf}

Let $\mathbb S^3(R)$ be the $3$-sphere of radius $R$ in $\ce^2\equiv \re^4$, let $\mathbb{S}^2(R/2)$ be the $2$-sphere of radius $R/2$ in $\re^3$, and let $\pi_R:\mathbb{S}^3(R)\rightarrow\mathbb{S}^2(R/2)$ be the Hopf fibration
\begin{align*}
\pi_R(z,w)=\frac{1}{2R}\left(2z\overline w,|z|^2-|w|^2\right),\qquad(z,w)\in\mathbb{S}^3(R)\subset\mathbb{C}^2.
\end{align*}
When $R=1$, we will omit the subindex $R$. We shall denote by $N$ the unit vector orthogonal to $\mathbb{S}^3(R)$ pointing inward. If $J$ is the natural complex structure of $\ce^2$, then the fibers of the Hopf fibration are the integral curves of $JN$, which are geodesics of $\mathbb{S}^3(R)$. For every $p\in \mathbb S^3(R)$, the subspace $\mathcal H_p=\{JN_p\}^\bot$ of $T_p \mathbb S^3(R)$ orthogonal to $JN_p$ is called horizontal, and it is invariant under $J$. Moreover $\pi_{R*}$ restricted to $\mathcal H$ is an isometry and, through this isometry, $J$ induces on $\mathbb S^2(R/2)$ a complex structure that we shall denote again by $J$.

Let $P_R$ be a  closed curve in $\mathbb S^2(R/2)$ which we will parametrize by $\gR(v)$,
 $v\in [0,2 \pi]$, where
$
\gR:\mathbb S^1\equiv[0,2\pi]/\!\sim \flecha \mathbb S^2(R/2)
$,
and define $M_R\subset\mathbb{S}^3(R)$ the Riemannian surface $\pi_R^{-1}(P_R)$ given by its position vector $F_R$ in 
$\mathbb{S}^3(R)$;
we remark that $F_R : \mathbb S^1 \times \mathbb S^1 \flecha \mathbb S^3(R)$. We have the following diagram: 
\begin{equation*}
\xymatrix{{M_R\equiv\mathbb{S}^1 \times \mathbb{S}^1 }\ar[rr]^{F_R} \ar[d]_{\pi_R} &&\mathbb{S}^3(R)\ar[d]^{\pi_R}\ar@{^{(}->}[rr]&&\mathbb{C}^2\\
P_R\equiv \mathbb{S}^1  \ar[rr]^{\gamma_R}&&\mathbb{S}^2(R/2)&&
}
\end{equation*}

If $X$ is a vector field tangent to $P_R$ or $\mathbb S^2(R/2)$, then $X^*$ will denote its horizontal lift tangent to $M_R$ or $\mathbb S^3(R)$, respectively. 

 Given $x\in \mathbb S^2(R/2)$, if we can write $\gamma_R$ as the image of a curve $c(u)$ in $T_x\mathbb S^2(R/2)$ by the exponential map $\exp_x$, then we can parametrize $F_R$ as $F_R(\beta,u) = \exp_{e^{i\beta} q} c(u)^*$, with $\pi(q)=x$ and $c(u)^*$ the horizontal lift of $c(u)$ at $e^{i\beta}q$.

Given $X,Y$ vector fields tangent to $\mathbb S^2(R/2)$ we get that
\begin{align}  
\pi_{R*}(\oN_{X^*} Y^*) = \hN_XY, \label{CDH}\end{align}
where $\oN$ and $\hN$ denote the Riemannian connections of $\mathbb S^3(R)$ and $\mathbb S^2(R/2)$, respectively. Moreover, we shall denote by $\wN$ the covariant derivative (that is, the standard directional derivative)  in $\ce^2$.

Let us denote by $e_1 ={\gamma_R'(u)}/{|\gamma_R'(u)|}$, by $\vec{\kappa}_R$ the curvature vector of $\gamma_R$ in $\mathbb S^2(R/2)$, by $\sigma_R$ and $\overline \sigma_R$  the second fundamental forms of $M_R$ in $\ce^2$ and $\mathbb S^3(R)$, respectively.  $H_R$ and $\oH_R$ will denote the respective mean curvatures. Moreover, $\widetilde\sigma$ will denote the second fundamental form  of $\mathbb S^3(R)$ in $\ce^2$.

One has that $\<\oN_{e_1^*} e_1^*, JN\> = - \<e_1^*, \oN_{e_1^*}JN\> = - \<e_1^*, \wN_{e_1^*}JN\> = \frac1R \<e_1^*, Je_1^*\> = 0$. That is, $\oN_{e_1^*} e_1^*$ is horizontal and, since it must be orthogonal to $e_1^*$, it is in the direction of $Je_1^*$.

As a consequence of this fact and \eqref{CDH} one has
\begin{align}
\vec{\kappa}_R = (\hN_{e_1} e_1)^\bot = \pi_{R*} ((\oN_{e_1^*} e_1^*)^{\bot}) = \pi_{R*} (\overline \sigma(e_1^*, e_1^*)).
\end{align}
That is, $\overline\sigma(e_1^*, e_1^*) = \vec{\kappa}_R^*$. Then, for the mean curvatures one has 
\begin{align}
H_R = \sigma_R(e_1^*,e_1^*) + \sigma_R(JN,JN) &=\overline\sigma_R(e_1^*,e_1^*) + \widetilde\sigma_R(e_1^*,e_1^*) + \overline\sigma_R(JN,JN)+ \widetilde\sigma(JN,JN) \nonumber \\
&= \vec{\kappa}_R^* + \frac{2}{R} N = \frac1{R} \vec{\kappa}^* - \frac2{R} F, \label{HR}
\end{align} 
where we have used, for the last equality, that under an homothety the curvature of a curve becomes divided by the magnitude of the  homothety. Recall also that when we consider $R=1$ we do not write the subindex $R$. Moreover, we have chosen $N$ pointing inward, which gives $N=-F$. 

\subsection{Spherical Lagrangian submanifolds}
In the complex Euclidean plane $\mathbb C^2$ we consider the bilinear
Hermitian product defined by
\[
(z,w)=z_1\bar{w}_1+z_2\bar{w}_2, \quad  z,w\in\mathbb C^2.
\]
Then $\langle\, \, , \, \rangle = {\rm Re} (\,\, , \,)$ is the
Euclidean metric on $\mathbb C^2$ and $\omega = -{\rm Im} (\,,)$ is the
Kaehler two-form given by $\omega (\,\cdot\, ,\,\cdot\,)=\langle
J\cdot,\cdot\rangle$, where $J$ is the complex structure on
$\mathbb C^2$.

Let $F:M \rightarrow\mathbb C^2$ be an isometric immersion of a surface $M$ into $\mathbb C^2$. $F $ is said to be Lagrangian if
$F^* \omega = 0$.  This is equivalent to the orthogonal decomposition $ T\mathbb C^2 =F_* TM \oplus
J F_* T M$, where $TM$ is the tangent bundle of $M$.
 
\begin{prop}\label{MHC}
Let $M$ be any compact Lagrangian surface of $\ce^2$ contained in some hypersphere $\mathbb S^3(R)$, $R>0$. Then $M$ must be the preimage of a closed curve in $\mathbb S^2(R/2)$ by the Hopf fibration  $\pi_R:\mathbb S^3(R) \rightarrow \mathbb S^2(R/2) $.
\end{prop} 
\begin{proof}
Let $N$ be the unit vector normal to $\mathbb S^3(R)$ in $\ce^2$. Then $JN$ is a vector field on $\mathbb S^3(R)$ whose integral curves are the fibres of the Hopf fibration $\pi_R$. Since $M$ is Lagrangian, the restriction of $JN$ to $M$ is a tangent vector field on $M$ and its integral curves are contained in $M$. In this way, the restriction of $\pi_R$ to $M$ is a Riemannian submersion on its image $\pi_R(M)=:C$ with the same fibres that the Hopf fibration. That is, $M = \pi_R^{-1}(C)$ for some closed curve $C\subset \mathbb S^2(R/2)$.
\end{proof}
\begin{nota}\label{param} 
{\rm If $F_R$ denotes  a Lagrangian  immersion of $M$ into  $\mathbb S^3(R)\subset \ce^2$, then Proposition \ref{MHC} tells us that $F_R$ can be regarded as $F_R :\mathbb S^1 \times \mathbb S^1 \flecha \mathbb S^3(R)$ and there exists a curve $\gR:\mathbb S^1  \flecha \mathbb S^2(R/2)$ such that $(\pi_R\circ F_R)(u,v) = \gR(v)$ and $F_R(\mathbb S^1\times\{v_0\})$ is a fibre of the Hopf fibration for every $v_0\in \mathbb S^1$.  That is, a Lagrangian spherical immersion of a compact surface is a Hopf torus and viceversa. There are many Lagrangian tori in $\mathbb{C}^2$ which are not Hopf tori; e.g., see~\cite[Section 3]{CL14}. }\end{nota}

\section{Proof of Theorem \ref{TheoremA}}\label{TA}

Let $M_t$ be a one-parameter family of Lagrangian surfaces of $\ce^2$ contained in the spheres $\mathbb S^3(R(t))$ of radius $R(t)>0$. Using Proposition~\ref{MHC} and Remark~\ref{param}, this family can be parametrized in the following way:
\begin{align} \label{FRt}
F_{R(t)}(u,v,t) = R(t) F(u,v,t),
\end{align} 
where  $F(\cdot,\cdot,t):\mathbb S^1 \times \mathbb S^1 \flecha \mathbb S^3$ is a family of Lagrangian immersions of a torus in $\ce^2$ contained in the unit hypersphere, and there exists a family of curves $\gamma(\cdot,t):\mathbb S^1  \flecha \mathbb S^2(1/2)$ such that $(\pi_{R(t)}\circ F_{R(t)})(u,v,t) = R(t) \gamma(v,t)$, which is equivalent to
\begin{align}\label{gammat}
(\pi\circ F)(u,v,t) =  \gamma(v,t).
\end{align}

\begin{lema}\label{lema-conditions}
The family of Lagrangian immersions given in \eqref{FRt} satisfies the mean curvature flow equation \eqref{MCF} in $\ce^2$ if and only if $R(t)=\sqrt{R_0^2-4t}$ and $F(\cdot,t)$ is the preimage by the Hopf fibration $\pi:\mathbb S^3 \flecha \mathbb S^2(1/2)$ of a curve $\gamma(\cdot,\ot(t))$ satisfying the mean curvature flow  equation in the 2-sphere, with the change of parameter given by $\ot=\frac{1}{4}\ln\frac{R_0^2}{R_0^2 - 4 t}$, where $R_0=R(0)$.
\end{lema}

\begin{proof}
The left side of \eqref{MCF} is obviously
\begin{equation}\label{parcialt}
\parcial{\FR}{t}(u,v,t) = R'(t)\ F(u,v,t) + R(t) \parcial{F}{t}(u,v,t).
\end{equation}
To compute the right side of \eqref{MCF}, we will use \eqref{HR} at each time $t$. Using \eqref{HR} and \eqref{parcialt}, the evolution equation $\partial F_R / \partial t = H_R$ becomes
\begin{align*}
R' F + R  \parcial{F}{t} = \frac1R \vec{\kappa}_\gamma^* - \frac2R F.
\end{align*}
Since $|F|=1$, necessarily $\parcial{F}{t} $ is orthogonal to $F$, and so the above equation separates in two coupled ones:
\begin{align*}
\begin{cases}
R' =  - \frac2R \\
R \parcial{F}{t} =\frac1R \vec{\kappa}_\gamma^*
\end{cases}
\end{align*}

Putting $R(0)=R_0$, the solution of the first equation is $R^2(t) = R_0^2 - 4 t$. Plugging this solution in the second one, we obtain that
\begin{equation}\label{dFat}
  \parcial{F}{t} =\frac1{R_0^2 - 4 t} \vec{\kappa}_\gamma^*.
\end{equation}
Using \eqref{gammat}, the composition with $\pi_*$ of the above equation implies that
\begin{equation}\label{dgat}
  \parcial{\g}{t} = \frac1{R_0^2 - 4 t} \vec{\kappa}_\gamma.
\end{equation}
This is not exactly the mean curvature flow for $\gamma(v,t)$; but we consider the change of parameter $t=t(\ot)$ given by
\begin{equation}\label{ot}
\ot = \ot (t) =  \int_0^t \frac1{R_0^2 - 4 s} ds = -\frac14 \ln \frac{R_0^2 - 4 t}{R_0^2} = \ln \(\frac{R_0^2}{R_0^2 - 4 t}\)^{1/4}.
\end{equation}
In this way, we arrive at 
\begin{equation}\label{evolFR}
\parcial{\g}{\ot} =  \parcial{t}{\ot}\frac1{R_0^2 - 4 t} \vec{\kappa}_\gamma = \vec{\kappa}_\gamma ,
\end{equation}
which is the mean curvature flow for $\gamma(u,t(\ot))$. 
\end{proof}

Next we employ Lemma~\ref{lema-conditions} to prove the following result. In particular, we deduce that the spherical condition is preserved by the Lagrangian mean curvature flow.

\begin{teor}\label{TF} 
Let $F_{R_0}$ be a Lagrangian immersion of a surface in $\ce^2$, contained in  the hypersphere $\mathbb S^3(R_0)$ of radius $R_0>0$. Then $F_{R_0}$ evolves under the mean curvature flow following the formula:
\begin{equation}\label{evexpFR}
F_{R_0}(\cdot,t) = \sqrt{R_0^2 - 4 t}\ F(\cdot, t),
\end{equation}
where $F(\cdot,t)$ is the preimage by the Hopf fibration $\pi:\mathbb S^3 \flecha \mathbb S^2(1/2)$ of a curve $\gamma(\cdot,\ot(t))$ satisfying the evolution equation \eqref{evolFR}, where  $\ot(t)$ is the function given in \eqref{ot}.
\end{teor}

 \begin{proof}
 Define $F_0=(1/R_0) F_{R_0}$, and $\gamma_0: \mathbb S^1 \flecha \mathbb S^2(1/2)$ satisfying $\pi\circ F_0 = \gamma_0$ as in Remark \ref{param}. Let $\gamma (\cdot,\overline t):\mathbb S^1 \times [0,\overline T[\flecha \mathbb S^2(1/2)$ be a solution of the curve shortening problem \eqref{evolFR} satisfying $\gamma(\cdot,0) = \gamma_0(\cdot)$. After the reparametrization of time given by \eqref{ot}, the family $\gamma(\cdot, t)$ is a solution of \eqref{dgat}. Then $F(\cdot,\cdot,t):\mathbb S^1 \times \mathbb S^1 \flecha \mathbb S^3$ defined as the family of Lagrangian surfaces of $\ce^2$ contained in $\mathbb S^3$ which are liftings of the Hopf fibration is a solution of \eqref{dFat} satisfying $F(\cdot,0)= F_0$, and $\sqrt{R_0^2 - 4 t} \ F(\cdot,t)$ is (by Lemma~\ref{lema-conditions})  a solution of the mean curvature flow equal to $F_{R_0}$ at $t=0$. By the uniqueness of the solution of the mean curvature flow with given initial condition, the statement of the theorem follows.
\end{proof}

In order to continue with the proof of Theorem \ref{TheoremA}, we need the following lemma.

\begin{lema}\label{dAt} 
Let $\gamma_0$ be a closed simple curve in $\mathbb S^2(1/2)$ enclosing a domain with area $A_0\leq \pi/2 $.
If $A(\ot)$ denotes the area enclosed by a solution $\gamma(\cdot,\ot)$ of \eqref{evolFR} with initial condition $\gamma(\cdot,0)=\gamma_0(\cdot)$, then $ A(\ot) = \pi/2 - \( \pi/2- A_0\) e^{4 \ot}$, and the extinction time of $\gamma(\cdot,\ot)$ is given by $\tau = \ln\(\frac{ \pi}{ \pi - 2 A_0}\)^{1/4} \leq \infty$.
\end{lema}
\begin{proof}
It is well known that the rate at which the area $A(\ot)$ decrease with time $\ot$ is given by $\partial A / \partial \ot = - \int_{\gamma} \kappa_\gamma \, \mathrm{d}s$, which implies using the Gauss-Bonnet formula that $A'(\ot) = 4 A(\ot) - 2\pi$, taking into account that $\gamma$ lies in a sphere of radius $1/2$. Solving the former equation, we obtain that $\ln \(2 \pi - 4 A(\ot)\)^{1/4} = \ln\(2\pi - 4 A_0\)^{1/4} + \ot$, and this proves the statement.
\end{proof}

\begin{coro}\label{33} Under the hypotheses of Theorem \ref{TF} and Lemma \ref{dAt}, there are only two possibilities for the evolution under the mean curvature flow of a Lagrangian embedding $F_{R_0}$ of a compact surface in $\ce^2$: 
\begin{itemize}
\item[(a)] If $F_{R_0}(\mathbb S^1\times\mathbb S^1)$ divides $\mathbb S^3(R_0)$ in two connected components of equal volume, 
then $F_{R_0}(\cdot, t)$ is defined for $t\in[0,R_0^2 / 4 )$, the limit of $F_{R_0}(\cdot, t)$ when $t\to R_0^2 / 4$ is the center of $\mathbb S^3(R_0)$, and rescaling $t$ by $\ot$ according to \eqref{ot} and $F_{R_0}(\cdot, t)$ by  $\widetilde F_{R_0}(\cdot,t)= \frac1{\sqrt{R_0^2 - 4 t} }F_{R_0}(\cdot,t)$, then $\lim_{\ot\to\infty} \widetilde F_{R_0}(\cdot,\ot)$ is the Clifford torus in $\mathbb S^3$.
\item[(b)] If $F_{R_0}(\mathbb S^1\times\mathbb S^1)$ divides $\mathbb S^3(R_0)$ in two connected components of different volumes, 
then $F_{R_0}(\cdot, t)$ is defined for $t\in[0,T) $, $T=A_0 R_0^2/2\pi<R_0^2 / 4$, and the limit of $F_{R_0}(\cdot, t)$ when $t\to T$ is a circle of radius $\sqrt{R_0^2 - 4 T}=R_0 \sqrt{1-2\pi/A_0} >0$, where $A_0$ is the area enclosed by the curve $\gamma_0 (\cdot)=\gamma(\cdot,0) \subset \mathbb S^2(1/2)$.
\end{itemize}
\end{coro}
\begin{proof} 
From Theorem \ref{TF} it follows that the flow $F_{R_0}(\cdot, t)$ given in \eqref{evexpFR} is defined in $[0,T )$, the intersection of the intervals where $\sqrt{R_0^2 - 4 t}$ and $\gamma(\cdot,\ot(t))$ are defined. 
On the one hand, this implies immediately that $T\le R_0^2 / 4$. On the other hand, $\gamma(\cdot,\ot)$ is well defined on  $[0, \tau )$ (see Lemma \ref{dAt}). Using \eqref{ot}, we get that
\begin{equation}\label{tot}
t(\tau) = \frac{R_0^2}{4} \(1- e^{-4 \tau}\).
\end{equation}
It is well known for the curve shortening flow in the 2-sphere (see for instance \cite{Ga90} and also \cite{CZ01}) that there are only two possibilities:
\begin{itemize}
\item[(a)] $\tau = \infty$ and $\lim_{\ot\to\infty}\gamma(\cdot,\ot)$ is a geodesic of $\mathbb S^2(1/2)$.

This case corresponds to $A_0=\pi /2$. Then it follows from \eqref{tot} that $t(\infty) = R_0^2 /4$ and so $\lim_{t\to R_0^2/4} \g(\cdot, \ot(t)) $ is a geodesic in $\mathbb S^2(1/2)$. Thus, the limit of the preimage $F(\cdot,t)$ when $t\to R_0^2 /4$ is the preimage of a geodesic in $\mathbb S^2(1/2)$, which is the Clifford torus in $\mathbb S^3$. Therefore, rescaling $t$ to get $\ot$ and $F_{R_0}(\cdot,t)$ to $\widetilde F_{R_0}(\cdot,t)= \frac1{\sqrt{R_0^2 - 4 t} }F_{R_0}(\cdot,t)$, we obtain that
\[
\widetilde F_{R_0}(\cdot,\ot) := \widetilde F_{R_0}(\cdot,t(\ot))= F(\cdot,t(\ot))
\]
and, as we have just deduced, $\lim_{\ot\to\infty} F(\cdot,t(\ot))$ is the Clifford torus in $\mathbb S^3$. 

\item[(b)]  $\tau <\infty$ and $\lim_{\ot\to\tau}\gamma(\cdot,\ot)$ is a point of $\mathbb S^2(1/2)$. 

This case corresponds to $A_0<\pi /2$. Using Lemma~\ref{dAt} and \eqref{tot}, we have that $T=t(\tau)=A_0 R_0^2/2\pi < R_0^2 / 4 $. Moreover, the limit when $t\to T$ of $\gamma(\cdot,\ot(t))$ is a point of $\mathbb S^2(1/2)$, whose preimage is a circle of radius $1$ in $\mathbb S^3$. Thus $\lim_{t\to T} F_{R_0}(\cdot,t)$ is a circle of radius $\sqrt{R_0^2 - 4 T} >0$ in $\mathbb S^3(\sqrt{R_0^2 - 4 T})$. \qedhere
\end{itemize}
\end{proof}

In the case (a) of Corollary~\ref{33} we have used the total space to rescale. However, in the case (b) we will use the base space to rescale. A natural rescaling for the curve $\gamma$ in $\mathbb S^2(1/2)$ shrinking to a point $x\in \mathbb S^2(1/2)$ is to consider the 2-sphere in $\re^3$ and to multiply $\gamma-x$ by a function of $\ot$ such that the area enclosed by the rescaled curves be constant. According to Lemma \ref{dAt}, this rescaling is given by
\begin{equation}\label{rescG}
\widetilde\gamma(\cdot,\ot) -x = \sqrt{\frac{A_0}{\pi/2 - \( \pi/2- A_0\) e^{4 \ot}}} \(\gamma(\cdot,t(\ot)) - x\).
\end{equation}
Now a well known result on the curve shortening flow in a surface (see \cite{Zh98}) implies that the limit of the rescaling \eqref{rescG} when $\ot\to \tau$ (that is, $t\to T$) is a planar circle centered at $x$ of radius $\sqrt{A_0/\pi}$.

Hence, taking into account the formula given in  \eqref{ot}, for the Lagrangian surface $F_{R_0}(\cdot, t)$ we will use the rescaling

\begin{align}\label{rescF}
\widetilde F_{R_0}(\cdot,t) -  R(t)q= \sqrt{\frac{A_0}{\pi/2 - \( \pi/2- A_0\) \(\frac{R_0^2}{R_0^2 - 4 t}\)}} \(F_{R_0}(\cdot, t) - R(t) q\)
\end{align}
where $R(t)=\sqrt{R_0^2-4 t}$ and  $q$ is a point in the limit circle of $F$ when $t\to T$.  Notice that $ \pi(q)= x$ and that the rescaling factor in \eqref{rescF} coincides with that in \eqref{rescG} when we consider the relation \eqref{ot}.

\begin{prop}\label{limfib} 
When $T<R_0^2/4$,
the limit of the rescaling \eqref{rescF} when $t\to T$ is a cylinder passing through $\sqrt{R_0^2-4 T} q$, which is the product of a circle of radius $\sqrt{(R_0^2-4 T)A_0/\pi} $ and a line.
\end{prop}
\begin{proof}
Let us denote 
\begin{equation}\label{lambda}
\lambda\equiv\lambda(t) :=  \sqrt{\frac{A_0}{\pi/2 - \( \pi/2- A_0 \) \(\frac{R_0^2}{R_0^2 - 4 t}\)}}.
\end{equation}
 We remark that $\lambda \to \infty$ when $t\to T= A_0 R_0^2 /2 \pi$ and recall that $R(t)= \sqrt{R_0^2- 4 t}$.

 The rescalings $\widetilde  F_{R_0}$ and $\widetilde \gamma_{R_0}:= R(t) \widetilde \gamma$, of $F_{R_0}$ and $\gamma_{R_0}:= R(t) \gamma$ respectively (see equations \eqref{rescF} and \eqref{rescG}), are just the restrictions to $F_{R_0}$ and $\gamma_{R_0}$ of the maps 
\begin{align*}
\mu_t &: \ce^2 \flecha \ce^2 ; \quad \mu_t(z) = R(t) q+ \lambda (z-R(t) q), \quad \text{ and } \\
 \nu_t&: \re^3 \flecha \re^3 ; \quad \nu_t(w) = R(t) x + \lambda \left(w- R(t) x\right),
 \end{align*} 
 which transform spheres in the following way:
 \begin{align*}
 \mu_t(\mathbb S^3_{R(t)}) = {\mathbb S^3}\! \left((1-\lambda) R(t) q, \lambda R(t) \right), \text{ and }\nu_t(\mathbb S^2_{R(t)/2}) = {\mathbb S^2}\! \left((1-\lambda) R(t) x, \lambda R(t)/2 \right),
 \end{align*}
 where $\mathbb S^m(y,r)$ indicates a sphere in $\re^{m+1}$ of radius $r$ and center $y$. Then the map 
 \begin{align*} 
 \widetilde{\pi_t} &=  \nu_t \circ \pi_{R(t)} \circ {\mu_t}^{-1} : {\mathbb S^3}\! \left((1-\lambda) R(t) q, \lambda R(t) \right) \flecha 
{\mathbb S^2}\! \left((1-\lambda) R(t) x, \lambda R(t)/2 \right) 
\end{align*}
is a Hopf fibration which, for every $z\in \mathbb S^3 \subset\ce^2$, takes the geodesic circle $(1-\lambda) R(t) q + \lambda \ R(t)e^{i \beta}z \in \mathbb S^3((1-\lambda) R(t) q, \lambda R(t))$ into the point $(1-\lambda) R(t)x + \lambda R(t)  \pi(e^{i\beta} z)\in \mathbb S^2((1-\lambda) R(t)x, \lambda R(t)/2)$. Moreover,  $\widetilde \pi_t \circ\widetilde F_{R_0}(\cdot,t) = \nu_t\circ \pi_{R(t)}\circ \mu_t^{-1} \circ \widetilde F_{R_0}(\cdot,t) = \nu_t\circ \pi_{R(t)} \circ F_{R_0}(\cdot,t) = \nu_t \circ \gamma_{R_0}(\cdot,t) = \widetilde \gamma_{R_0}(\cdot,t)$. Moreover, since $\widetilde F_{R_0}(\cdot,t)$ is a Lagrangian submanifold of $\ce^2$, Proposition~\ref{MHC} applies to the Hopf map $\widetilde \pi_t$ and so  $\widetilde F_{R_0}(\cdot,t)$ is the preimage of $\widetilde \gamma_{R_0}(\cdot, t)$ by $\widetilde \pi_t$.

Let $\{e_1,e_2\}$ be an orthonormal basis of $T_{R(t) x} \mathbb S^2((1-\lambda) R(t)x, \lambda \ R(t)/2)$, and let $\{e_1^*,e_2^*\}$ be its corresponding lifting to the fiber on $R(t)x$ in  $\mathbb S^3((1-\lambda) R(t) q, \lambda \ R(t))$. Since $\widetilde \gamma_{R_0}(\cdot, t)$ converges to a circle with center at $R(T) x$ and radius $R(T) \sqrt{A_0/\pi}$, when $t$ is near $T$, $\widetilde \gamma_{R_0}(\cdot, t)$ becomes convex near its limit, and can be parametrized in the form 
\begin{align*}
\widetilde \gamma_{R_0}(\varphi, t) &= \exp_{R(t)x} r(\varphi,t) (\cos \varphi \ e_1 +  \sin\varphi \ e_2) \nn \\
&= (1-\lambda ) R(t) x + \frac\lambda{2} R(t) \left(2x \cos \frac{r(\varphi,t)}{\lambda R(t)} + (\cos \varphi e_1 +  \sin\varphi e_2) \sin \frac{r(\varphi,t)}{\lambda R(t)}\right)
\end{align*} 
with $\lim_{t\to T}r(\varphi,t) = R(T) \sqrt{A_0/\pi}$, where $\exp$ denotes the exponential map in $\mathbb S^2((1-\lambda) R(t)x, \lambda R(t)/2)$.
As a consequence, since $\widetilde F_{R_0}(\cdot,t)$ is the preimage of $\widetilde \gamma_{R_0}(\cdot, t)$ by $\widetilde\pi_t$, it can be parametrized (as was recalled in section \ref{hopf}) by
\begin{equation}
\begin{aligned}\label{torot}
\widetilde F_{R_0}(\varphi, s ,t) &= \exp_{(1-\lambda)R(t)q + e^{i(s/(\lambda R(t)))} \lambda R(t) q} r(\varphi,t) (\cos \varphi e_1^* + \sin\varphi e_2^*) \\
&= (1-\lambda ) R(t) q + \lambda R(t) \left(\left(q \cos \frac{s}{\lambda R(t)} + Jq \sin \frac{s}{\lambda R(t)}\right) \cos \frac{r(\varphi,t)}{\lambda R(t)} \right. \\
&  \left.\qquad \qquad \qquad \qquad \qquad + (\cos \varphi e_1^* +  \sin\varphi e_2^*) \sin \frac{r(\varphi,t)}{\lambda R(t)}\right),
\end{aligned}
\end{equation}
where $s$ is the arclength of the curve $(1-\lambda)R(t)q + e^{i(s/(\lambda R(t)))} \lambda R(t) q$.

Now, taking the limit in \eqref{torot}  when $t\to T$ (which implies $\lambda\to \infty$), we obtain the cylinder 
\begin{align*}
\widetilde F_{R_0}(\varphi, s ,T) &=  R(T) q + s Jq + (\cos \varphi e_1^* +  \sin\varphi e_2^*)  R(T) \sqrt{A_0/\pi},
\end{align*}
which is the cylinder indicated in the statement of the Proposition. \end{proof}

\begin{nota}\label{resSN}
{\rm We observe that the rescaling \eqref{rescF} is not exactly the standard one given in \cite{Hu84}. Nevertheless, they only differ in the product by a bounded function and consequently they are equivalent.   Thus the blow up will be again a cylinder in $\re^3\subset \ce^2$ }.
\end{nota}

\begin{nota}\label{notsingI}
{\rm All the singularities appearing in Theorem \ref{TheoremA} are Type I singularities.
In fact, following Section 2 and using Theorem~\ref{TF}, it is not difficult to check that the second fundamental form $\sigma$ of the evolution \eqref{evexpFR} is given by
\begin{align} \label{sigka} | \sigma|^2 = \frac{4+\kappa_\gamma^2}{R_0^2-4t}, \qquad t\in [0,T). 
\end{align} 
In case (a), we have that $T=R_0^2/4$ and we know that $\kappa_\gamma$ is bounded by some constant $L$; then we get that $(T-t)|\sigma|^2 = 1+\kappa_\gamma^2/4 \le 1+L/4 $, which implies the condition of being a Type I singularity.

In case (b), we have that $T=t(\tau) = A_0 R_0^2 /2 \pi < R_0^2 /4$ and we know that $\gamma$ develops a Type I singularity. So there exists a constant $C$ such that $(\tau - \ot)\kappa_\gamma^2 \leq C$. Using that $A_0<\pi/2$ and \eqref{ot}, we get that
$$(T-t)|\sigma|^2 < 1+ \frac{(T-t)\kappa_\gamma^2}{R_0^2-4t}= 1+\frac{1-e^{4(\ot-\tau)}}{4}\kappa_\gamma^2 .$$ 
If we define $G(\ot)=\bigl(1-e^{4(\ot-\tau)}\bigr)/4-(\tau - \ot)$, it is easy to check that $G'(\ot)>0$ and so $G(\ot)<G(\tau)=0$. Hence we conclude that $(T-t)|\sigma|^2 < 1+(\tau - \ot)\kappa_\gamma^2 \leq 1+C$,
that shows that the behaviour of $|\sigma|$ in case (b) is determined by the one of $|\kappa_\gamma|$, which corresponds to a Type I singularity.}
\end{nota}

\author{Ildefonso Castro, icastro@ujaen.es \\
Departamento de Matem\'{a}ticas \\
Universidad de Ja\'{e}n \\
23071 Ja\'{e}n, Spain 
}
\vspace{0.1cm}

\author{Ana M.~Lerma, alerma@ujaen.es \\
Departamento de Did\'{a}ctica de las Ciencias \\
Universidad de Ja\'{e}n \\
23071 Ja\'{e}n, Spain 
}
\vspace{0.1cm}

\author{Vicente Miquel, miquel@uv.es \\
Departamento de Geometr\'{\i}a y Topolog\'{\i}a \\
Universidad de Valencia \\
46100-Burjassot (Valencia), Spain 
}

\end{document}